\documentclass[12pt,a4paper]{amsart}

\usepackage[margin=3cm]{geometry}
\usepackage{t1enc}
\usepackage[utf8]{inputenc}
\usepackage{amsthm,amsmath,amssymb}
\usepackage{graphicx}
\usepackage{enumerate}
\usepackage{hyperref}
\usepackage{bm}

\usepackage{amsfonts}
\usepackage{graphicx}
\usepackage{bm}
\usepackage{amsmath, amsthm, amssymb}
\usepackage{graphicx}
\usepackage{hyperref}
 \usepackage{relsize}

\usepackage{bbm}

\theoremstyle{plain}

\newtheorem{theorem}{Theorem}
\newtheorem{lemma}[theorem]{Lemma}
\newtheorem{corr}[theorem]{Corollary}

\theoremstyle{definition}

\title{A BK inequality for random matchings}
\author{Andr\'as M\'esz\'aros}

\address{Central European University, Budapest\and\newline\indent Alfr\'ed R\'enyi Institute of Mathematics, Budapest}

\email{Meszaros\_Andras@phd.ceu.edu}

\begin{document}

\maketitle

 \begin{abstract}
Let $G=(S,T,E)$ be a bipartite graph.   For a matching $M$ of $G$, let $V(M)$ be the set of vertices covered by $M$, and let $B(M)$ be the symmetric difference of $V(M)$ and $S$. We prove that if $M$ is a uniform random matching of~$G$, then $B(M)$ satisfies the BK inequality for increasing events.   

 \end{abstract}




\section{Introduction}

Let $V$ be a finite set. We will consider random subsets of $V$. Let $\mathcal{A}$ and $\mathcal{B}$ be upward closed subsets of $2^V$, in other words, let $\mathcal{A}$ and $\mathcal{B}$ be increasing events. Let $\mathcal{A}\square\mathcal{B}$ be the event that $\mathcal{A}$ and $\mathcal{B}$ both occur disjointly, more formally, we define

\[\mathcal{A}\square \mathcal{B}=\{A\cup B|\quad A\in\mathcal{A},B\in \mathcal{B},A\cap B=\emptyset\}.\]

Let $G=(S,T,E)$ be a bipartite graph, and let $V=S\cup T$.  Let $\mathcal{M}$ be the set of matchings in $G$. For a matching $M\in \mathcal{M}$, let $V(M)$ be the set of vertices covered by $M$, and let \[B(M)=V(M)\Delta S,\] where $\Delta$ denotes the symmetric difference. Note that we have $|B(M)|=|S|$ for any matching~$M$. 

Our main result is the following.

\begin{theorem}\label{thm1}
Let $M$ be a uniform random element of $\mathcal{M}$. Then $B(M)$ satisfies the BK inequality for increasing events, that is, if $\mathcal{A}$ and $\mathcal{B}$ are upward closed subsets of~$2^V$, then
\[\mathbb{P}(B(M)\in \mathcal{A}\square \mathcal{B} )\le \mathbb{P}(B(M)\in\mathcal{A})\mathbb{P}(B(M)\in\mathcal{B}).\]
\end{theorem}

For a random subset with independent marginals, the BK inequality was proved by van den Berg and Kesten \cite{bk}. Later, van den Berg and  Jonasson proved that it also holds for a uniform random $k$ element subset \cite{bj}. There is an extension of the notion  $\mathcal{A}\square \mathcal{B}$ for arbitrary events, see Subsection \ref{gendef}.  With this definition, the BK inequality holds for all events in the case of  a random subset with independent marginals. This was conjectured by van den Berg and Kesten \cite{bk}, and  proved by Reimer \cite{reimer}. See also the paper of van den Berg and  Gandolfi~\cite{bg} for further results.

We say that an event $\mathcal{A}$  depends only on $V_0\subseteq V$, if for any $A,B\subseteq V$ the conditions $A\cap V_0=B\cap V_0$ and $A\in \mathcal{A}$ imply that $B\in \mathcal{A}$. Note that if $\mathcal{A}$ and $\mathcal{B}$ are increasing events depending on disjoint subsets of $V$, then $\mathcal{A}\square\mathcal{B}=\mathcal{A}\cap \mathcal{B}$. Thus, Theorem \ref{thm1} has the following corollary.

\begin{corr}\label{corna}
Let $B(M)$ be like above, then $B(M)$ has negative associations, which means the following. Let $\mathcal{A}$ and $\mathcal{B}$ be events depending on disjoint subsets of $V$. If $\mathcal{A}$ and $\mathcal{B}$ are both increasing or both decreasing, then
\[\mathbb{P}(B(M)\in \mathcal{A}\cap \mathcal{B} )\le \mathbb{P}(B(M)\in\mathcal{A})\mathbb{P}(B(M)\in\mathcal{B}).\]
If $\mathcal{A}$ is increasing and $\mathcal{B}$ is decreasing, then
\[\mathbb{P}(B(M)\in \mathcal{A}\cap \mathcal{B} )\ge \mathbb{P}(B(M)\in\mathcal{A})\mathbb{P}(B(M)\in\mathcal{B}).\] 
\end{corr}

Now we give a few extensions of Theorem \ref{thm1}. Assume that every edge~$e$ of $G$ has positive weight $w(e)$. For a matching $M$, we define the weight of $M$ as \break $w(M)=\prod_{e\in M} w(e)$. Let $M$ be a random matching, where the probability of a matching is proportional to its weight. We have the following extension of \break Theorem~\ref{thm1}.

\begin{theorem}\label{thm2}
Let $M$ be like above. Then $B(M)$ satisfies the BK inequality for increasing events, that is, if $\mathcal{A}$ and $\mathcal{B}$ are upward closed subsets of $2^V$, then
\[\mathbb{P}(B(M)\in \mathcal{A}\square \mathcal{B} )\le \mathbb{P}(B(M)\in\mathcal{A})\mathbb{P}(B(M)\in\mathcal{B}).\] 
\end{theorem} 

Furthermore, let $V_+$ and $V_-$ be disjoint subsets of $V$. Let $M'$ have the same distribution as $M$ conditioned on the event that $V_+\subseteq B(M)$ and $V_-\cap B(M)=\emptyset$. Let $V'=V\backslash(V_+\cup V_-)$, and let $B'(M')=B(M')\cap V'$. Clearly, $B'(M')$ is a random subset of $V'$.
\begin{theorem}\label{thm3}
The random subset $B'(M')$ satisfies  the BK inequality for increasing events.
\end{theorem}

This has the following corollary.
\begin{corr}\label{corsubm}
Let $M$ be like above. Then for any subset $X$ and $Y$ of $V$, we have
\[\mathbb{P}(X\subseteq B(M))\mathbb{P}(Y\subseteq B(M))\ge \mathbb{P}(X\cap Y\subseteq B(M))\mathbb{P}(X\cup Y\subseteq B(M)).\]
\end{corr}

We can also deduce the following theorem from Theorem \ref{thm2}.
\begin{theorem}\label{thm4}
Let $M$ be uniform random maximum size matching.  Then the random subset $B(M)$ satisfies  the BK inequality for increasing events.
\end{theorem}

\textbf{Acknowledgements.} The author is grateful to P\'eter  Csikv\'ari and Mikl\'os Ab\'ert for their comments. The author was partially
supported by the ERC Consolidator Grant 648017.



\section{The proofs}

\subsection{The definition of $\mathcal{A}\square\mathcal{B}$ for arbitrary events}
\label{gendef}
Let us recall how to extend the definition of $\mathcal{A}\square\mathcal{B}$ to arbitrary events. A subset $C$ of $V$ is in $\mathcal{A}\square \mathcal{B}$ if and only if there are disjoint subsets $V_A$ and $V_B$ of $V$ such that
\[\{D\subseteq V| D\cap V_A=C\cap V_A \}\subseteq \mathcal{A}\]
and 
\[\{D\subseteq V| D\cap V_B=C\cap V_B \}\subseteq \mathcal{B}.\]
If $\mathcal{A}$ and $\mathcal{B}$ are increasing, then this definition indeed coincides with our earlier definition.

\subsection{The proof of Theorem \ref{thm1}}

Our proof will use several ideas of Berg and Jonasson \cite{bj}.

Let $I$ be the set of tuples $(W,K,L,R)$, where $W$ is a subset of $V$, $K$ and $L$ are perfect matchings in the induced  subgraph $G[W]$, $R$ is a subgraph of $G[V\backslash W]$ consisting of vertex disjoint paths.\footnote{In our terminology, a path must have at least $1$ edge.} 

Fix a linear ordering of the edges of $G$. Consider an $i=(W,K,L,R)\in I$. Then $R$ is the vertex disjoint union of the paths $P_1,P_2,\dots,P_k$, where we list the paths in increasing order of their lowest edge. We can write $P_j$ as the union of the matchings $M_{j,0}$ and $M_{j,1}$, this decomposition is unique once we assume that $M_{j,0}$ contains the lowest edge of $P_j$. For $\omega=(\omega_1,\omega_2,\dots,\omega_k)\in \{0,1\}^k$, we define the matchings
\[C_{i,\omega}=K\cup\cup_{j=1}^k M_{j,\omega_j}\quad\text{and}\quad D_{i,\omega}=L\cup\cup_{j=1}^k M_{j,1-\omega_j}.\]      
Moreover, we define
\[Y_i^C=\{C_{i,\omega}|\quad \omega\in\{0,1\}^k\},\]
\[Y_i^D=\{D_{i,\omega}|\quad \omega\in\{0,1\}^k\},\]
and
\[X_i=\{(C_{i,\omega},D_{i,\omega})|\quad \omega\in\{0,1\}^k\}.\] 

Let $H_i$ be the set of endpoints of the paths $P_1,P_2,\dots,P_k$. Let $V(R)$ be the vertex set of $R$. Let $B_i=((W\cup V(R))\Delta S)\backslash H_i$. Let $v_{j,0}$ and $v_{j,1}$ be the two endpoints of~$P_j$. If we choose the indices in the right way, then we get that
\[B(C_{i,\omega})=B_i\cup\{v_{j,\omega_j}|\quad j=1,2,\dots,k\},\] 
and
\[B(D_{i,\omega})=B_i\cup\{v_{j,1-\omega_j}|\quad j=1,2,\dots,k\}.\] 

This immediately implies that 
\begin{multline}\label{eq1}
\{B(C_{i,\omega})|\quad \omega\in \{0,1\}^k\}=\{B(D_{i,\omega})|\quad \omega\in \{0,1\}^k\}=\\
\{B_i\cup H|\quad H\subseteq H_i\text{ and } |H\cap\{v_{j,0},v_{j,1}\}|=1\text{ for all }j=1,2,\dots,k\}.
\end{multline}

Let $U=\{v_{j,1}|\quad j=1,2,\dots,k\}$. We define the map $\tau_i:\mathcal{M}\to 2^U$ by \break $\tau_i(M)=B(M)\cap U$. It is clear from what is written above that the appropriate restriction of $\tau_i$ gives  a bijection from $Y_i^C$ to $2^U$, and also from $Y_i^D$ to $2^U$. Moreover,  
\begin{equation}
X_i=\{(C,D)\in Y_i^C\times Y_i^D|\quad \tau_i(C)=U\backslash \tau_i(D)\}. 
\end{equation}
\begin{lemma}\label{lemma1}
The sets $(X_i)_{i\in I}$ give a partition of $\mathcal{M}\times \mathcal{M}$.
\end{lemma}
\begin{proof}
Let $(C,D)\in \mathcal{M}\times \mathcal{M}$. Consider the multi-graph $C\cup D$, it is a vertex disjoint union of cycles and paths. Let $R$ be the union of paths, and let $Q$ be the union of cycles. Let $W$ be the vertices covered by the cycles. Let $i=(W,C\cap Q,D\cap Q,R)$. One can easily prove that $i$ is the unique element of $I$ such that $(C,D)\in X_i$.
\end{proof}

Given a subset $\mathcal{F}$ of $2^V$, we define $\mathcal{M}_{\mathcal{F}}$ as $\{M\in \mathcal{M}| B(M)\in \mathcal{F}\}$. 
The statement of Theorem \ref{thm1} is equivalent to the statement
\[|\mathcal{M}_{\mathcal{A}\square \mathcal{B}}\times \mathcal{M}|\le |\mathcal{M}_{\mathcal{A}}\times \mathcal{M}_{\mathcal{B}}|.\]

From Lemma \ref{lemma1}, it follows that it is enough to prove that for any $i\in I$, we have
\begin{equation}\label{partineq}
|(\mathcal{M}_{\mathcal{A}\square \mathcal{B}}\times \mathcal{M})\cap X_i|\le |(\mathcal{M}_{\mathcal{A}}\times \mathcal{M}_{\mathcal{B}})\cap X_i|.
\end{equation}

For a subset $\mathcal{F}$ of $2^V$ and $i\in I$, 
we define $\mathcal{F}^i=\{\tau_i(C)|C\in Y_i^C\cap\mathcal{M}_\mathcal{F}\}$. 
From~(\ref{eq1}), it follows that 
$\mathcal{F}^i=\{\tau_i(D)|D\in Y_i^D\cap\mathcal{M}_\mathcal{F}\}$. (Note that, even for  an increasing $\mathcal{F}$ it might happen that $\mathcal{F}^i$ is not increasing.)   For a subset $\mathcal{J}$ of $2^U$, we define $\overline{\mathcal{J}}=\{U\backslash J|J\in\mathcal{J}\}$.

Then 
\begin{align}\label{line1}
|(\mathcal{M}_{\mathcal{A}}\times&  \mathcal{M}_\mathcal{B})\cap X_i|\\&=
|\{(C,D)\in Y_i^C\times Y_i^D| \tau_i(C)\in\mathcal{A}^i,\tau_i(D)\in\mathcal{B}^i,\tau_i(C)=U\backslash\tau_i(D) \}|\nonumber\\&=
|\{(A,B)\in 2^U\times 2^U|A\in\mathcal{A}^i,B\in\mathcal{B}^i,A=U\backslash B \}|\nonumber\\&=|\mathcal{A}^i\cap \overline{\mathcal{B}^i}|.\nonumber
\end{align}

Similarly,
\begin{equation}\label{line2}
|(\mathcal{M}_{\mathcal{A}\square \mathcal{B}}\times \mathcal{M})\cap X_i|=|(\mathcal{A}\square\mathcal{B})^i|.
\end{equation}

\begin{lemma}\label{lemmatart}
We have
\[(\mathcal{A}\square \mathcal{B})^i\subseteq \mathcal{A}^i\square \mathcal{B}^i.\]
\end{lemma}
\begin{proof}
Let $F\in (\mathcal{A}\square \mathcal{B})^i$, then $F=\tau_i(C)$ for some $C\in Y_i^C$ such that $B(C)\in \mathcal{A}\square \mathcal{B}$. Since $\mathcal{A}$ and $\mathcal{B}$ are upward closed, there are disjoint sets $V_A\in \mathcal{A}$ and $V_B\in\mathcal{B}$ such that $B(C)=V_A\cup V_B$. We define
\[U_A=\{v_{j,1}|\quad\{v_{j,0},v_{j,1}\}\cap V_A\neq\emptyset, j\in\{1,2,\dots,k\}\}\]
and
\[U_B=\{v_{j,1}|\quad\{v_{j,0},v_{j,1}\}\cap V_B\neq\emptyset, j\in\{1,2,\dots,k\}\}.\]
Since $V_A$ and $V_B$ are disjoint and $|B(C)\cap \{v_{j,0},v_{j,1}\}|=1$ for all $j$, we obtain that $U_A$ and $U_B$ are disjoint. 

Moreover, if for some $C'\in Y_i^C$, we have $\tau_i(C)\cap U_A=\tau_i(C')\cap U_A$, then $ V_A\subseteq B(C')$, consequently $B(C')\in\mathcal{A}$ and $\tau_i(C')\in \mathcal{A}^i$. The analogous statement is true for $V_B$ and $U_B$. Therefore, the pair $U_A,U_B$ witnesses that $F=\tau_i(C)\in \mathcal{A}^i\square \mathcal{B}^i$.       
\end{proof}

Recall the following theorem of Reimer \cite{reimer}. See also \cite{bj}.
\begin{theorem}[Reimer]\label{rthm}
Let $\mathcal{X}$ and $\mathcal{Y}$ be subsets of $2^U$, where $U$ is a finite set. Then
\[|\mathcal{X}\square \mathcal{Y}|\le |\mathcal{X}\cap \overline{\mathcal{Y}}|.\]
\end{theorem}

Combining Theorem \ref{rthm} with Equations (\ref{line1}) and (\ref{line2}) and Lemma \ref{lemmatart}, we obtain that 
\begin{multline*}
|(\mathcal{M}_{\mathcal{A}\square \mathcal{B}}\times \mathcal{M})\cap X_i|=|(\mathcal{A}\square \mathcal{B})^i|\le|\mathcal{A}^i\square \mathcal{B}^i| \le|\mathcal{A}^i\cap \overline{\mathcal{B}^i}|=|(\mathcal{M}_{\mathcal{A}}\times  \mathcal{M}_\mathcal{B})\cap X_i|.
\end{multline*}
This proves Inequality (\ref{partineq}).







\subsection{The proof of Theorem \ref{thm2}}

Consider an $i\in I$. Observe that $w(C)\cdot w(D)$ is the same for any $(C,D)\in X_i$. Thus, it is again enough to prove Inequality (\ref{partineq}), so the whole proof goes through.

\subsection{The proof of Theorem \ref{thm3}}  

We define 
\[\mathcal{M}'=\{M\in \mathcal{M}|\quad V_+\subseteq B(M), V_-\cap B(M)=\emptyset\}.\]
Recall that for $i=(W,K,L,R)\in I$, we defined $H_i$ as the endpoints of the paths in~$R$, and $B_i$ as $B_i=((W\cup V(R))\Delta S)\backslash H_i$. Now we define
\[I'=\{i\in I|\quad V_+\subseteq B_i, V_-\cap(B_i\cup H_i)=\emptyset\}.\]

Using the following lemma, the proof of Theorem \ref{thm1} can be repeated again.
\begin{lemma}
The sets $(X_i)_{i\in I'}$ give a partition of $\mathcal{M}'\times \mathcal{M}'$.
\end{lemma}
\begin{proof}
The proof is almost identical to that of Lemma \ref{lemma1}.
\end{proof}
\subsection{The proof Corollary \ref{corsubm}}

 Let $X_0=X\backslash Y$ and $Y_0=Y\backslash X$. Clearly the events $X_0\subseteq B(M)$ and $Y_0\subseteq B(M)$ depend on disjoint sets. Theorem \ref{thm3} gives us
\begin{multline*}
\mathbb{P}(X_0\subseteq B(M)|X\cap Y\subseteq B(M)) \mathbb{P}(Y_0\subseteq B(M)|X\cap Y\subseteq B(M)) \\\ge \mathbb{P}(X_0\subseteq B(M),Y_0\subseteq B(M)|X\cap Y\subseteq B(M)),
\end{multline*}
and this is equivalent with the statement of the corollary.    

\subsection{The proof Theorem \ref{thm4}}

Let $t>0$, and set all the edge weights to be equal to~$t$. Let $M_t$ be the corresponding random matching. By Theorem \ref{thm2}, if $\mathcal{A}$ and $\mathcal{B}$ are increasing events, then
\[\mathbb{P}(B(M_t)\in \mathcal{A}\square \mathcal{B} )\le \mathbb{P}(B(M_t)\in\mathcal{A})\mathbb{P}(B(M_t)\in\mathcal{B}).\]
Observe that 
\begin{align*} \lim_{t\to\infty}\mathbb{P}(B(M_t)\in \mathcal{A})&=\mathbb{P}(B(M)\in \mathcal{A}),\quad \lim_{t\to\infty}\mathbb{P}(B(M_t)\in \mathcal{B})=\mathbb{P}(B(M)\in \mathcal{B})\\&\text{ and }\lim_{t\to\infty}\mathbb{P}(B(M_t)\in \mathcal{A}\square \mathcal{B})=\mathbb{P}(B(M)\in \mathcal{A}\square \mathcal{B}).
\end{align*}
Thus, the statement follows.

\bibliography{bkref}
\bibliographystyle{plain}

\end{document}